\documentclass[11pt,leqno,twoside]{amsart}
\usepackage{amssymb,amsmath,amsthm,soul,color}
\usepackage{t1enc}
\usepackage[cp1250]{inputenc}
\usepackage{a4,indentfirst,latexsym,}
\usepackage{graphics}
\usepackage{mathrsfs}
\usepackage{cite,enumitem,graphicx}
\usepackage[colorlinks=true,urlcolor=blue,
citecolor=red,linkcolor=blue,linktocpage,pdfpagelabels,
bookmarksnumbered,bookmarksopen]{hyperref}
\usepackage[english]{babel}
\usepackage[left=2.61cm,right=2.61cm,top=2.72cm,bottom=2.72cm]{geometry}
\usepackage[metapost]{mfpic}
\usepackage[hyperpageref]{backref}
\usepackage[colorinlistoftodos]{todonotes}
\usepackage[normalem]{ulem}

\makeatletter
\providecommand\@dotsep{5}
\def\listtodoname{List of Todos}
\def\listoftodos{\@starttoc{tdo}\listtodoname}
\makeatother

\numberwithin{equation}{section}
\newtheorem{theorem}{Theorem}[section]
\newtheorem{proposition}[theorem]{Proposition}

\newtheorem{lemma}[theorem]{Lemma}

\newtheorem{remark}{Remark}

\def\N{\mathbb{N}}
\def\R{\mathbb{R}}

\def\R {{\rm I}\hskip -0.85mm{\rm R}}
\def\N {{\rm I}\hskip -0.85mm{\rm N}}

\def\ol{\overline}
\def\ul{\underline}

\title[Nonlocal problems with singular nonlinearity]
{
  Positive solutions for  a class of nonlocal problems \\ 
  with possibly singular nonlinearity
}

\author[L. Gasi\'nski]{Leszek  Gasi\'nski}
\author[J. R. Santos Jr.]{Jo\~ao R. Santos Junior}
\author[G. Siciliano]{Gaetano Siciliano }

\address[L. Gasi\'nski]
{\newline\indent Department of Mathematics
\newline\indent
Pedagogical University of Cracow
\newline\indent
Podchorazych 2, 30-084 Cracow, Poland}
\email{\href{mailto: leszek.gasinski@up.krakow.pl }{leszek.gasinski@up.krakow.pl}
}

\address[J. R. Santos Jr.]
{\newline\indent Faculdade de Matem\'atica
\newline\indent
Instituto de Ci\^{e}ncias Exatas e Naturais
\newline\indent
Universidade Federal do Par\'a
\newline\indent
Avenida Augusto corr\^{e}a 01, 66075-110, Bel\'em, PA, Brazil}
\email{\href{mailto: joaojunior@ufpa.br }{joaojunior@ufpa.br}
}
\address[G. Siciliano]
{\newline\indent Departamento de Matem\'atica
\newline\indent
Instituto de Matem\'atica e Estat\'istica
\newline\indent
 Universidade de S\~ao Paulo
\newline\indent
Rua do Mat\~ao 1010,  05508-090, S\~ao Paulo, SP, Brazil }
\email{\href{mailto:sicilian@ime.usp.br}{sicilian@ime.usp.br}
}

\thanks{Jo\~ao R. Santos was partially
supported by CNPq 306503/2018-7, Brazil.
Gaetano Siciliano  was partially supported by
Fapesp 2018/17264-4 and 2016/23746-6, Capes and CNPq 304660/2018-3, Brazil. }

\begin{document}

\maketitle
\begin{abstract}
  We study a class of elliptic problems with homogeneous
  Dirichlet boundary condition and a nonlinear
  reaction term $f$ which is nonlocal depending on the $L^{p}$-norm of the unknown function.
  The nonlinearity $f$ can make the problem degenerate since it may even
  have multiple singularities  in the nonlocal variable.
  We use  fixed point arguments for an appropriately defined solution map,
  to produce multiplicity of classical positive solutions with ordered norms.
\end{abstract}

\bigskip

\maketitle
\begin{center}
\begin{minipage}{12cm}
\tableofcontents
\end{minipage}
\end{center}

\bigskip

\section{Introduction}


  In this paper we consider the existence of multiple positive solutions for the following
class of nonlocal problems:
\begin{equation}\label{P1}\tag{$P_{I}$}
\left \{ \begin{array}{ll}
  -\Delta u = f\left(u, \displaystyle\int_{\Omega} | u|^{p}dx \right) & \mbox{in $\Omega$,}\\
  u=0 & \mbox{on $\partial\Omega$,}
\end{array}\right.
\end{equation}
  where $p\geq1$, $\Omega$ is a bounded domain in $\R^{N}$ with smooth boundary
  and  $f$ is a continuous function with a  suitable  behaviour in the second variable.

\medskip

  In the last decades partial differential equations involving nonlocal terms have attracted a great deal of attention
  of the mathematical community for different reasons. Indeed, equations involving nonlocal terms are usually more realistic
  to model different situations in nature, see Furter-Grinfeld \cite{FG} for a comparison between local and nonlocal models in population dynamics,
  or Kirchhoff \cite{Kir} for an improvement (based on the insertion of a nonlocal term) of the classical D'Alembert's wave equation
  in string deformation theory. From a mathematical point of view
  nonlocal equations are challenging since, in general, the presence of a nonlocal term
  makes the equation much more complicated. In many cases the known techniques
  cannot be applied in a straightforward way, so the development of alternative approaches is required.

Regarding stationary partial differential equations in very recent years two classes of problems involving nonlocal terms in the diffusion operator have been a quite active research field, namely the Kirchhoff and the generalised Carrier problems. A basic prototype of the first one is
\begin{equation}\label{P3}\tag{$KP$}
\left \{ \begin{array}{ll}
  -a\left(\displaystyle\int_{\Omega}|\nabla u|^{2}dx\right)\Delta u = f(u) & \mbox{in $\Omega$,}  \medskip \\
  u=0 & \mbox{on $\partial\Omega$,}
\end{array}\right.
\end{equation}
  where $a$ is usually a continuous real function bounded away from zero.
  Problem \eqref{P3} is the stationary version of a hyperbolic model to small transversal vibrations
  of elastic membranes, see Kirchhoff \cite{Kir} for details.
  On the other hand, generalised Carrier problems are of type
\begin{equation}\label{P4}\tag{$GCP$}
\left \{ \begin{array}{ll}
  -a\left(\displaystyle\int_{\Omega}|u|^{p}dx\right)\Delta u = f(u) & \mbox{in $\Omega$,} \medskip \\
  u=0 & \mbox{on $\partial\Omega$,}
\end{array}\right.
\end{equation}
  where $a$ is again a continuous real function. When $p=2$, problem \eqref{P4} is the stationary case
  of another hyperbolic model of transversal vibrations (not necessarily small)
  of elastic membranes which was introduced by Carrier in \cite{Car}, in the unidimensional case.
  As $p=1$, \eqref{P4} also appears as the stationary version of parabolic models
  in the study of the dispersion of biological populations where the diffusion rate
  of the specie takes into account the whole population. See, for instance, Chipot-Rodrigues \cite{CR}.
  For papers involving operators as in \eqref{P4} with $p>1$ arbitrary,
  we refer the reader to Figueiredo-Sousa-Morales-Rodrigo-Su\'arez \cite{FMS} and references therein.

 \medskip

  Recently, the above problems have been studied also in the degenerate case, that is,
  when $a$ vanishes in some points.
  Concerning \eqref{P3}, this line of research started with the papers of
  Ambrosetti-Arcoya \cite{AA1,AA2}.
  See also Santos J\'unior-Siciliano \cite{SS} who establish a multiplicity result of positive solutions
  depending on the number of degeneration points of the function $a$.
  For problem \eqref{P4},
  some interesting results can be found in the recent papers of
  Delgado-Morales-Rodrigo-Santos J\'unior-Su\'arez \cite{DMSS} and
  Gasi\'nski-Santos J\'unior \cite{GS1,GS2}.

\medskip

  Starting with the above considerations, our aim here is to study existence of positive
  solutions for (possibly) degenerate problems for which problem 
  \eqref{P4} is just a particular case.
  The main goal of this paper is to prove that when the reaction term involved in the equation depends on the 
  $L^{p}$-norm of the unknown function and it allows singularities in some points of its domain,
  the same sort of multiplicity result (as those observed in degenerate Kirchhoff
  and Carrier problems, see
  Santos J\'unior-Siciliano \cite{SS} and
  Gasi\'nski-Santos J\'unior \cite{GS1}) holds.

  It is worth to point out that such kind of results are much more general than those obtained for degenerate
  Carrier problems. In fact, if we choose, the nonlinearity $f$
  in \eqref{P1} of the form
$$
  f\left(u, \int_{\Omega} |u|^{p}dx \right )=\frac{g(u)}{a\left(\displaystyle\int_{\Omega} |u|^{p}dx\right)},
$$
  with $a$ being a continuous real function which vanishes in some points,
  then $f(u,\cdot)$ exhibits a diverging behaviour and we fall in the degenerate 
  Carrier problems \eqref{P4}.

\subsection{Statement of the main result - Theorem \ref{main1}} \label{sec:Prelim}

  First let us introduce some notations.
  Along the paper,
  \begin{itemize}
  \item $\|\cdot\|$ denotes the $H_{0}^{1}(\Omega)$-norm,
  \item   $|\cdot|_{r}$ is the $L^{r}(\Omega)$-norm,
  \item   $\lambda_{1}$ is the first eigenvalue of the minus Laplacian operator in $\Omega$
  with zero Dirichlet boundary condition,
  \item   $\varphi_{1}$ is the positive eigenfunction associated to $\lambda_{1}$ normalized in the $H_{0}^{1}(\Omega)$-norm,
  \item $e_{1}$ is the positive eigenfunction associated to $\lambda_{1}$ normalized in the $L^{\infty}(\Omega)$-norm,
  \item $C_{1}>0$ stands for best constant of the Sobolev embedding of $H_{0}^{1}(\Omega)$ into $L^{1}(\Omega)$;
\item  $|\Omega|$ is the Lebesgue measure of $\Omega$.
  \end{itemize}

  \medskip

We assume that $f$ is a function having the following behaviour:

\begin{enumerate}[label=($f_{\arabic*}$),ref=$f_{\arabic*}$,start=0]
\item\label{f_{0}} there exist positive numbers $0=:t_{0}<t_{1}<t_{2}<\ldots<t_{K}$ ($K\geq 1$) such that
     $f:\R\times \mathcal{A}\to\R$ is continuous, where $\mathcal{A}=(0,\infty)\backslash\{t_{1}, \ldots, t_{K}\}$.
     For each $k\in\{1, \ldots, K\}$ and fixed $\alpha\in (t_{k-1}, t_{k})$,
     $f(\cdot,\alpha)\in C^1(\R)$,
     there exists $s_\alpha>0$ such that $f(s_{\alpha}, \alpha)=0$ and
     $f(s, \alpha)>0$ for $s\in (0, s_{\alpha})$.
  Moreover $\sup\limits_{\alpha\in Z}s_{\alpha}<\infty$, for each compact set $Z\subset (t_{k-1}, t_{k})$ and $k\in\{1, \ldots, K\}$;
 \item\label{f_{1}} for each $k\in\{1, \ldots, K\}$ and any $s\in (0,s_{\alpha})$, $\lim\limits_{\alpha\to t_{k}}f(s, \alpha)\in[\lambda_{1},+\infty]$;
 \item\label{f_{2}} for each $k\in\{1, \ldots, K\}$ and fixed $\alpha\in (t_{k-1}, t_{k})$, the map $(0, s_{\alpha})\ni s\mapsto f(s, \alpha)/s$ is decreasing;
 \item\label{f_{3}}$\inf\limits_{\alpha\in \mathcal{A}}\gamma_{\alpha}>\lambda_1$, where $\gamma_\alpha:=\lim\limits_{s\to 0^{+}}f(s, \alpha)/s$.
\end{enumerate}


\begin{remark} \label{rem_hyp123}\rm
  \textbf{(a)} Hypothesis \eqref{f_{1}} makes the problems \eqref{P1} 
   singular in the nonlocal term
  whenever $\lim\limits_{\alpha\to t_{k}}f(s, \alpha)=+\infty$.\\
  \textbf{(b)} Hypotheses  \eqref{f_{2}} and \eqref{f_{3}} describe the behaviour of the function
  $\psi_{\alpha}\colon (0, s_{\alpha})\to (0, \gamma_{\alpha})$
  defined by
\begin{equation}\label{eq_psi}
  \psi_{\alpha}(s)=\frac{f(s,\alpha)}{s},\ \forall s\in (0, s_{\alpha}).
\end{equation}
  Hypothesis \eqref{f_{2}} guarantees its monotonicity, which will be crucial
  in the proof of the uniqueness of solutions for the auxiliary problem \eqref{AP}
  (see Proposition \ref{pro1}) and for other technical issues (see Lemmas \ref{odd}, \ref{odt}, and \ref{behhhh}).
  Hypothesis \eqref{f_{3}} says that the limits of the functions $\psi_{\alpha}$ at zero are separated from $\lambda_1$ (uniformly in $\alpha$).
\end{remark}

  Moreover we assume that:
  \medskip

\begin{enumerate}[label=($f_{\arabic*}$),ref=$f_{\arabic*}$,start=4]
 \item\label{f_{4}} $t_{K}<\big(\inf\limits_{\alpha\in\mathcal{A}}s_{\alpha}\big)^{p}\displaystyle\int_{\Omega}e_{1}^{p}dx$;
 \item\label{f_{5}} $\inf\limits_{\alpha\in (t_{k-1}, t_{k})}\max\limits_{s\in[0, s_{\alpha}]}f(s,\alpha)(s_{\alpha}^{p-1}/\alpha)<\lambda_{1}^{1/2}/C_{1}|\Omega|^{1/2}$,
             for all $k\in\{1, \ldots, K\}$.
\end{enumerate}

\medskip

%
%

\begin{remark}\label{rem_hyp45}\rm
  \textbf{(a)} Hypotheses \eqref{f_{4}} 
  gives some bounds on the ``edges'' of the
  curvilinear trapezoid $\{s_{\alpha}:\ \alpha\in\ \mathcal{A}\}\times \mathcal{A}$,
  in which the interesting domain of $f$ is contained.\\
  \textbf{(b)} Hypotheses \eqref{f_{5}} 
  is quite technical and  says that
  for some particular choice of $\alpha_0\in(t_{k-1}, t_{k})$ the maximum of the quotient $f(s,\alpha_0)/\alpha_0$
  over $s\in[0,s_{\alpha}]$ is bounded by some constant.
\end{remark}

  Our main result is the following.

\begin{theorem}\label{main1}
  If conditions \eqref{f_{0}}-\eqref{f_{5}} hold, then problem \eqref{P1} has at least $2K$ classical positive solutions with ordered $L^{p}$-norms, namely
$$
  0<\int_{\Omega}u_{1, 1}^{p}dx<\int_{\Omega}u_{1, 2}^{p}dx<t_{1}<\ldots<t_{K-1}<\int_{\Omega}u_{K, 1}^{p}dx<\int_{\Omega}u_{K, 2}^{p}dx<t_{K}.
$$
\end{theorem}


  The proof of the above theorem will follow the following general steps:
  \begin{itemize}
  \item[Step 1.] We first introduce an auxiliary problem with a truncated nonlinearity and fixed nonlocal term;
  \item[Step 2.] By variational methods we show that the auxiliary problem has a unique solution;
  \item[Step 3.] We define a suitable map which gives the $L^{p}$-norm
  of the solution;
  \item[Step 4.] Finally, we show that this map has fixed points, which are indeed solutions of the considered problem.
  \end{itemize}

  The structure of the paper is the following.
  In Section  \ref{sec:Auxiliary} we develop Step 1 and Step 2 above.
  In Section \ref{sec:Th1}  we address Step 3 and Step 4 above and implement them to the proof of Theorem \ref{main1}.
  In the final Section \ref{sec:Final} we give examples of nonlinearities $f$ satisfying our assumptions.

\section{Auxiliary problem}\label{sec:Auxiliary}

\subsection{Statement of the auxiliary problem \eqref{AP}}
  In order to prove our results, the following auxiliary problem will play an important role: for each $k\in\{1, \ldots, K\}$ and any $\alpha\in (t_{k-1}, t_{k})$ fixed, consider

\begin{equation}\label{AP}\tag{$P_{k, \alpha}$}
  \left \{ \begin{array}{ll}
    -\Delta u = \widehat{f}_{\alpha}(u) & \mbox{in $\Omega$,}\\
    u>0 & \mbox{in $\Omega$,}\\
    u=0 & \mbox{on $\partial\Omega$,}
  \end{array}\right.
\end{equation}
  where
\begin{equation*}
  \widehat{f}_{\alpha}(s)=\left \{ \begin{array}{ll}
  f(0, \alpha) & \mbox{if $s\leq 0$,}\\
  f(s, \alpha) & \mbox{if $0<s\leq s_{\alpha}$,}\\
  0 & \mbox{if $s_{\alpha}<s$.}\\
\end{array}\right.
\end{equation*}

\begin{remark}\label{rem1}\rm
  Since $f$ is continuous, it is an immediate consequence of the definition of $\widehat{f}_{\alpha}$
  that if $\{\alpha_{n}\}\subset (t_{k-1}, t_{k})$ converges to $\alpha\in (t_{k-1}, t_{k})$
  and $s_{n}\in (0, s_{\alpha_{n}})$ is such that $s_{n}\to s$, for some $s\in (0, s_{\alpha})$,
  then $\widehat{f}_{\alpha_{n}}(s_{n})\to\widehat{f}_{\alpha}(s)$.
\end{remark}

\subsection{Existence of the solutions of \eqref{AP}}

  Since we have ``removed'' the nonlocal term of \eqref{P1}, we can treat problem \eqref{AP} using a variational approach.

\begin{proposition}\label{pro1}
  If conditions \eqref{f_{0}}, \eqref{f_{2}} and \eqref{f_{3}} hold, then for each $k\in\{1, \ldots, K\}$ and $\alpha\in (t_{k-1}, t_{k})$ fixed, problem \eqref{AP} has a
  unique classical solution $0<u_{\alpha}\leq s_{\alpha}$.
\end{proposition}

\begin{proof}
  Since, for each $k\in\{1, \ldots, K\}$ and $\alpha\in (t_{k-1}, t_{k})$ fixed, $\widehat{f}_{\alpha}$ is bounded and continuous, the energy functional
$$
  I_{k, \alpha}(u)=\frac{1}{2}\|u\|^{2}-\int_{\Omega}\widehat{F}_{\alpha}(u)dx
$$
  corresponding to problem \eqref{AP} is coercive and weakly lower semicontinuous.
  Here as usual,
$$\widehat{F}_{\alpha}(s)=\int_{0}^{s}\widehat{f}_{\alpha}(\sigma) d\sigma.$$
  Therefore $I_{k, \alpha}$ has a minimum point $u_{\alpha}$
  which is a weak solution of \eqref{AP}.
  Moreover, it follows from conditions \eqref{f_{2}} and \eqref{f_{3}} that
$$
  \frac{I_{k, \alpha}(s\varphi_{1})}{s^{2}}=\frac{1}{2}-\int_{\Omega}\frac{\widehat{F}_{\alpha}(s\varphi_{1})}{(s\varphi_{1})^{2}}\varphi_{1}^{2}dx\to
  \frac{1}{2}\left(1-\frac{\gamma_{\alpha}}{\lambda_{1}}\right)<0 \ \mbox{as $s\to 0^{+}$}.
$$
  The last inequality implies that the minimum point $u_{\alpha}$ of $I_{k, \alpha}$ is nontrivial, because for $s>0$ small enough
$$
  I_{k, \alpha}(u_{\alpha})\leq I_{k, \alpha}(s\varphi_{1})=(I_{k, \alpha}(s\varphi_{1})/s^{2})s^{2}<0.
$$
  Then $u_{\alpha}\in H^{1}_{0}(\Omega)$ satisfies
\begin{equation}\label{eq:weaksolution}
  \int_{\Omega} \nabla u_{\alpha} \nabla v dx= \int_{\Omega} \widehat f_{\alpha}(u_{\alpha}) v dx \quad
  \forall\ v\in H^{1}_{0}(\Omega).
\end{equation}
  Then  it is easy to see that any nontrivial weak solution $u_{\alpha}$ of problem \eqref{AP}
  satisfies $0\leq u_{\alpha}\leq s_{\alpha}$. Indeed by choosing $v=(u_{\alpha}-s_{\alpha})^{+}$ in \eqref{eq:weaksolution} we arrive
  (by the definition of $\widehat f_{\alpha}$) at
$$
  \int_{\Omega}\nabla u_{\alpha}\nabla (u_{\alpha}-s_{\alpha})^{+} dx=\int_{\Omega} \widehat{f}_{\alpha}(u_{\alpha}) (u_{\alpha}-s_{\alpha})^{+} dx =0
$$
  which redly  implies $u_{\alpha}\leq s_{\alpha}$. Analogously one deduces that $0\leq u_{\alpha}$ by taking $v=u_{\alpha}^{-}$ in  \eqref{eq:weaksolution}.
  Morover
  the weak solution of problem \eqref{AP} is unique by condition \eqref{f_{2}},
  see Brezis-Oswald \cite{BO}.
  Since $\widehat{f}_{\alpha}(u_{\alpha})=f(u_{\alpha},\alpha)$ is bounded and $f(\cdot,\alpha)\in C^{1}(\R)$, it follows
  from Agmon \cite{Ag} that $u_{\alpha}$ is a classical solution. Finally, the maximum principle completes the proof
  (see  Gilbarg-Trudinger \cite[Theorem 3.1]{GT}).
\end{proof}

\medskip

\subsection{Properties of the solutions of \eqref{AP}}

  Due to Proposition \ref{pro1}, we can set
\begin{equation}\label{eq:ca}
  c_{\alpha}:=I_{k,\alpha}(u_{\alpha})=\min_{u\in H_{0}^{1}(\Omega)}I_{k, \alpha}(u).
\end{equation}

  Since, by condition \eqref{f_{2}}, the map $(0, s_{\alpha})\ni s\mapsto\psi_{\alpha}(s)=\widehat{f}_{\alpha}(s)/s$ is decreasing
  (see \eqref{eq_psi} in Remark \ref{rem_hyp123}),
  there exists the inverse $\psi_{\alpha}^{-1}\colon (0, \gamma_{\alpha})\to (0, s_{\alpha})$.
  Thereby, by condition \eqref{f_{3}}, for
  each $\varepsilon\in (0, \inf\limits_{\alpha\in\mathcal{A}}\gamma_{\alpha}-\lambda_{1})$, it makes sense to consider the function
$$
  y_{\alpha}:=\psi_{\alpha}^{-1}(\lambda_{1}+\varepsilon)e_{1}.
$$

\begin{lemma}\label{odd}
  If conditions \eqref{f_{0}}, \eqref{f_{2}} and \eqref{f_{3}} hold, $\alpha\in (t_{k-1}, t_{k})$, then for each
  $\varepsilon\in (0, \inf\limits_{\alpha\in\mathcal{A}}\gamma_{\alpha}-\lambda_{1})$, we have
\begin{equation}\label{7}
  c_{\alpha}\leq -\frac{1}{2}\varepsilon\psi_{\alpha}^{-1}(\lambda_{1}+\varepsilon)^{2}\int_{\Omega}e_{1}^{2} dx.
\end{equation}
  where $c_{\alpha}$ is given in \eqref{eq:ca}.
\end{lemma}

\begin{proof}
  Note that, by condition \eqref{f_{2}}, we have
$$
  \widehat{F}_{\alpha}(t)\geq \frac12\widehat{f}_{\alpha}(s)s\quad \forall \ s\geq 0.
$$
  Hence,
$$
  \frac{I_{k,\alpha}(y_{\alpha})}{\psi_{\alpha}^{-1}(\lambda_{1}+\varepsilon)^{2}}
  \leq\frac{1}{2}\left[\|e_{1}\|^{2}-\int_{\Omega}\frac{\widehat{f}_{\alpha}(y_{\alpha})}{\psi_{\alpha}^{-1}(\lambda_{1}+\varepsilon)^{2}}y_{\alpha}dx\right],
$$
  or equivalently
$$
  \frac{I_{k, \alpha}(y_{\alpha})}{\psi_{\alpha}^{-1}(\lambda_{1}+\varepsilon)^{2}}
  \leq\frac{1}{2}\left[\|e_{1}\|^{2}-\int_{\Omega}\frac{\widehat{f}_{\alpha}(y_{\alpha})}{y_{\alpha}}e_{1}^{2}dx\right].
$$
  Using the definition of $e_{1}$ and condition \eqref{f_{2}}, we get
$$
  \frac{I_{k,\alpha}(y_{\alpha})}{\psi_{\alpha}^{-1}(\lambda_{1}+\varepsilon)^{2}}
  \leq\frac{1}{2}\left[\|e_{1}\|^{2}-\int_{\Omega}\frac{\widehat{f}_{\alpha}(\psi_{\alpha}^{-1}(\lambda_{1}+\varepsilon))}{\psi_{\alpha}^{-1}(\lambda_{1}+\varepsilon)}e_{1}^{2}dx\right].
$$
  Now, using the definition of $\psi_{\alpha}^{-1}$, we deduce
$$
  \frac{I_{k,\alpha}(y_{\alpha})}{\psi_{\alpha}^{-1}(\lambda_{1}+\varepsilon)^{2}}
  \leq\frac{1}{2}\left[\|e_{1}\|^{2}-(\lambda_{1}+\varepsilon)\int_{\Omega}e_{1}^{2}dx\right]=-\frac{1}{2}\varepsilon\int_{\Omega}e_{1}^{2}dx.
$$
  Therefore,
$$
  c_{\alpha}\leq I_{k, \alpha}(y_{\alpha})\leq -\frac{1}{2}\varepsilon\psi_{\alpha}^{-1}(\lambda_{1}+\varepsilon)^{2}\int_{\Omega}e_{1}^{2} dx,
$$
  which concludes the proof.
\end{proof}

  The next two technical results  will be helpful in what follows.

\begin{lemma}\label{odt}
  If conditions \eqref{f_{0}}, \eqref{f_{2}} and \eqref{f_{3}} hold, then
\begin{equation}\label{2}
  u_{\alpha}\geq z_{\alpha}:=\psi_{\alpha}^{-1}(\lambda_{1})e_{1}\quad \forall \ \alpha\in (t_{k-1}, t_{k}).
\end{equation}
\end{lemma}

\begin{proof}

  From condition \eqref{f_{2}} and definition of $\psi_{\alpha}^{-1}$ it follows that
$$
  \lambda_{1}= \frac{\widehat{f}_{\alpha}(\psi_{\alpha}^{-1}(\lambda_{1}))}{\psi_{\alpha}^{-1}(\lambda_{1})}\leq \frac{\widehat{f}_{\alpha}(z_{\alpha})}{z_{\alpha}}.
$$
  Thus
$$
  -\Delta(z_{\alpha})=\lambda_{1}z_{\alpha}\leq \widehat{f}_{\alpha}(z_{\alpha}) \ \mbox{in $\Omega$}.
$$

  Therefore $z_{\alpha}$ is a subsolution of \eqref{AP}. Inequality \eqref{2} follows now from condition
  \eqref{f_{2}} and
  Ambrosetti-Brezis-Cerami \cite[Lemma 3.3]{ABC}.
\end{proof}

\begin{lemma}\label{behhhh}
  If conditions \eqref{f_{0}}-\eqref{f_{2}} hold, then
\begin{equation}\label{232}
  \liminf\limits_{\alpha\to t_{k-1}^{+}}\psi_{\alpha}^{-1}(\lambda_{1})\geq \inf\limits_{\alpha\in (t_{k-1}, t_{k})}s_{\alpha}.
\end{equation}
  The same holds for $\alpha\to t_{k}^{-}$.
\end{lemma}

\begin{proof}
 Assuming the contrary, there would exist a sequence $\{\alpha_n\}\subseteq (t_{k-1},t_k)$ such that
 $\alpha_n\to t_{k-1}^+$ (or alternatively $\alpha_n\to t_{k}^-$) and
\[
  \psi_{\alpha_n}^{-1}(\lambda_{1})<T<\inf\limits_{\alpha\in (t_{k-1}, t_{k})}s_{\alpha},
\]
  for some $T\in\R$. Since $T\in [0,s_{\alpha_n}]$ for all $n\ge 1$ and
  $\psi_{\alpha_n}$ is decreasing on $[0,s_{\alpha_n}]$, we would have
\[
  \lambda_1>\psi_{\alpha_n}(T)=\frac{f(T,\alpha_n)}{T}.
\]
  But this contradicts hypothesis \eqref{f_{1}}.
  So the proof is done.
\end{proof}

\section{Proof of the main result}\label{sec:Th1}

\subsection{Auxiliary map $\mathcal P_{k}$}

  In virtue of Proposition \ref{pro1}, given $p\geq1$, for any $k\in \{1,\ldots, K\}$ we can define the map
\begin{equation}\label{eq:Pk1}
  \mathcal{P}_{k}: \alpha\in (t_{k-1}, t_{k}) \longmapsto \int_{\Omega}u_{\alpha}^{p}dx\in \R
\end{equation}
  where $u_{\alpha}$ is the unique solution of problem \eqref{AP}.
  The strategy of proving Theorem \ref{main1} will be to show that $\mathcal P_{k}$ is continuous and has
  two fixed points.
  Indeed any fixed point, let us say $\overline\alpha$, of $\mathcal P_{k}$ satisfies by definition
\begin{equation*}
\left \{ \begin{array}{ll}
  -\Delta u_{\overline\alpha} = \widehat{f}\left(u, \displaystyle\int_{\Omega} u^{p}_{ \overline\alpha} \,dx \right) & \mbox{in $\Omega$,}\\
  0<u_{\overline \alpha}\leq s_{\overline\alpha} & \mbox{in $\Omega$,}\\
  u_{\overline \alpha}=0 & \mbox{on $\partial\Omega$,}
\end{array}\right.
\end{equation*}
  and hence $u_{\overline\alpha}$ is a solution of \eqref{P1} with $L^{p}$-norm in $(t_{k-1}, t_{k})$.

\subsection{Continuity of $\mathcal P_{k}$}
\begin{proposition}\label{pp2}
  If conditions \eqref{f_{0}}, \eqref{f_{2}} and \eqref{f_{3}} hold, then for each $k\in \{1, 2, \ldots, K\}$,
  the map $\mathcal{P}_{k}$ defined in \eqref{eq:Pk1}
 is continuous.
\end{proposition}

\begin{proof}
  Let $\{\alpha_{n}\}\subset (t_{k-1}, t_{k})$ be a sequence such that $\alpha_{n}\to \alpha_{\ast}$,
  for some $\alpha_{\ast}\in (t_{k-1}, t_{k})$. Denote by $u_{n}$ the positive solution of $\eqref{AP}$ with $\alpha=\alpha_{n}$. Since,
\begin{equation}\label{1}
  \frac{1}{2}\|u_{n}\|^{2}-\int_{\Omega}F_{\alpha}(u_{n})dx=I_{k, \alpha}(u_{n})<0,
\end{equation}
  we get
$$
\|u_{n}\|\leq \left[2F_{\alpha}(s_{\alpha})|\Omega|\right]^{1/2},\quad \forall \ n\in\N.
$$
  Therefore, $\{u_{n}\}$ is bounded in $H_{0}^{1}(\Omega)$ and, up to a subsequence, there exists $u_{\ast}\in H_{0}^{1}(\Omega)$ such that
\begin{equation}\label{19}
  u_{n}\rightharpoonup u_{\ast} \ \mbox{in} \ H_{0}^{1}(\Omega).
\end{equation}
  Thus, from Remark \ref{rem1}, passing to the limit as $n\to\infty$ in
$$
  \int_{\Omega}\nabla u_{n}\nabla v dx=\int_{\Omega} \widehat{f}_{\alpha_{n}}(u_{n})v dx\quad \forall \ v\in H_{0}^{1}(\Omega),
$$
we get
$$
  \int_{\Omega}\nabla u_{\ast}\nabla v dx=\int_{\Omega} \widehat{f}_{\alpha_{\ast}}(u_{\ast})v dx\quad \forall \ v\in H_{0}^{1}(\Omega).
$$
  So, $u_{\ast}$ is a nonnegative weak solution of \eqref{AP} with $\alpha=\alpha_{\ast}$.
  We are going to show that $u_{\ast}\neq 0$. In fact, passing to the limit as
  $n\to\infty$ in
$$
  \int_{\Omega}\nabla u_{n}\nabla u_{\ast} dx=\int_{\Omega} \widehat{f}_{\alpha_{n}}(u_{n})u_{\ast} dx
$$
  and
$$
  \|u_{n}\|^{2}=\int_{\Omega} \widehat{f}_{\alpha_{n}}(u_{n})u_{n} dx,
$$
  we conclude that
\begin{equation}\label{20}
  \|u_{n}\|\to \|u_{\ast}\|.
\end{equation}
  From \eqref{19} and \eqref{20}, we have
\begin{equation}\label{21}
  u_{n}\to u_{\ast} \ \mbox{in $H_{0}^{1}(\Omega)$}.
\end{equation}

  By Lemma \ref{odd}, there exists $\varepsilon>0$, small enough, such that
$$
  I_{k, \alpha_{n}}(u_{n})\leq -\frac{1}{2}\varepsilon\psi_{\alpha}^{-1}(\lambda_{1}+\varepsilon)^{2}\int_{\Omega}e_{1}^{2} dx\quad \forall \ n\in\N.
$$
  So, passing to the limit as $n\to\infty$ and using \eqref{21}, we obtain
$$
  I_{k, \alpha_{\ast}}(u_{\ast})\leq -\frac{1}{2}\varepsilon\psi_{\alpha}^{-1}(\lambda_{1}+\varepsilon)^{2}\int_{\Omega}e_{1}^{2} dx<0.
$$
  Therefore $u_{\ast}\neq 0$. Arguing as in the proof of Proposition \ref{pro1} we can show that $u_{\ast}$ is a positive classical solution of \eqref{AP} with
  $\alpha=\alpha_{\ast}$. Since such a solution is unique, we conclude that
\begin{equation}\label{eq:continuity}
  u_{\ast}=u_{\alpha_{\ast}}.
\end{equation}
  Consequently,
\begin{equation}\label{21b}
  -\Delta(u_{n}-u_{\ast})=\widehat{f}_{\alpha_{n}}(u_{n})-\widehat{f}_{\alpha_{\ast}}(u_{\ast})=:g_{n}(x)\quad \forall \ n\in\N.
\end{equation}
  From \eqref{f_{0}}, the continuity of $f$ and inequalities $0\leq u_{n}\leq s_{\alpha_{n}}, 0\leq u_{\ast}\leq s_{\alpha_{\ast}}$, there exists a  constant $C>0$, such that
\begin{equation}\label{22}
  |g_{n}|_{\infty}\leq C_2\quad \forall \ n\in\N.
\end{equation}
  It follows from \eqref{21b}, \eqref{22} and
  Theorem 0.5 of Ambrosetti-Prodi \cite{AmPr} that there exists $\beta\in (0, 1)$ such that
$$
  \|u_{n}-u_{\ast}\|_{C^{1,\beta}(\overline{\Omega})}\leq C'\quad \forall \ n\in\N,
$$
  for some $C'>0$. By the compactness of embedding from $C^{1,\beta}(\overline{\Omega})$ into $C^{1}(\overline{\Omega})$ and \eqref{21}, up to a
  subsequence, we have
\begin{equation}\label{23}
  u_{n}\to u_{\ast} \ \mbox{in $C^{1}(\overline{\Omega})$}.
\end{equation}
  Convergence in \eqref{23} and inequality
$$
  ||u_{n}|_{p}-|u_{\ast}|_{p}|\leq |u_{n}-u_{\ast}|_{p}\leq |\Omega|^{1/p}|u_{n}-u_{\ast}|_{\infty}
$$
  lead us to
$$
  \mathcal{P}_{k}(\alpha_{n})\to\mathcal{P}_{k}(\alpha_{\ast}).
$$
  This proves the continuity of $\mathcal{P}_{k}$.
\end{proof}

\subsection{Existence of fixed points of $\mathcal{P}_{k}$}

\begin{proposition}\label{pp3}
  If conditions \eqref{f_{0}}-\eqref{f_{5}} hold, then the map $\mathcal{P}_{k}$ defined by \eqref{eq:Pk1}
  has at least two fixed points $t_{k-1}<\alpha_{1, k}<\alpha_{2, k}<t_{k}$.
\end{proposition}

\begin{proof}

  We start with two claims describing the behaviour of $\mathcal{P}_{k}$.

\medskip

\noindent {\bf Claim 1:}
  $\lim\limits_{\alpha\to t_{k-1}^{+}}\mathcal{P}_{k}(\alpha)>t_{k-1}$ and $\lim\limits_{\alpha\to t_{k}^{-}}\mathcal{P}_{k}(\alpha)>t_{k}$.

\medskip

  From Lemma \ref{odt}, we have
$$
  \mathcal{P}_{k}(\alpha)\geq (\psi_{\alpha}^{-1}(\lambda_{1}))^{p}\int_{\Omega}e_{1}^{p}dx\quad \forall \ \alpha\in (t_{k-1}, t_{k}).
$$
  Hence, by condition \eqref{f_{4}} and Lemma \ref{behhhh}, we get
\begin{eqnarray*}
  \liminf\limits_{\alpha\to t_{k-1}^{+}}\mathcal{P}_{k}(\alpha) & \geq & \left(\inf\limits_{\alpha\in (t_{k-1}, t_{k})}s_{\alpha}\right)^{p}\int_{\Omega}e_{1}^{p}dx>t_{K}>t_{k-1},\\
  \liminf\limits_{\alpha\to  t_{k}^{-}}\mathcal{P}_{k}(\alpha) & \geq & \left(\inf\limits_{\alpha\in (t_{k-1}, t_{k})}s_{\alpha}\right)^{p}\int_{\Omega}e_{1}^{p}dx>t_{K}>t_{k}.
\end{eqnarray*}

\medskip

\noindent {\bf Claim 2:}
  There exists $\alpha_0\in (t_{k-1}, t_{k})$ such that $\mathcal{P}_{k}(\alpha_0)<\alpha_0$.

\medskip

  For each $\alpha\in (t_{k-1}, t_{k})$, let $w_{\alpha}$ be the unique solution (which is positive) of the problem
\begin{equation*}
  \left \{ \begin{array}{ll}
  -\Delta u = u_{\alpha}^{p-1} & \mbox{in $\Omega$,}\\
  u=0 & \mbox{on $\partial\Omega$,}
  \end{array}\right.
\end{equation*}
  where $u_{\alpha}$ is the unique positive solution of \eqref{AP}. Hence, multiplying by $u_{\alpha}$ and integrating by parts, we have
$$
  \int_{\Omega}\nabla w_{\alpha}\nabla u_{\alpha} dx=\int_{\Omega}u_{\alpha}^{p}dx=\mathcal{P}_{k}(\alpha).
$$
  On the other hand, by the definition of $u_{\alpha}$, we get
\begin{equation}\label{12}
  \mathcal{P}_{k}(\alpha)=\int_{\Omega}\widehat{f}_{\alpha}(u_{\alpha})w_{\alpha} dx
\end{equation}
  thus,
\begin{equation}\label{16}
  \mathcal{P}_{k}(\alpha)\leq\left(\max\limits_{s\in [0, s_{\alpha}]}\widehat{f}_{\alpha}(s)\right)C_{1}\|w_{\alpha}\|,
\end{equation}
  (where $C_{1}>0$ is the best constant of the Sobolev embedding from $H_{0}^{1}(\Omega)$ into $L^{1}(\Omega)$).

  From the definition of $w_{\alpha}$, the fact that $0<u_{\alpha}\leq s_{\alpha}$ and H\"older's inequality, we obtain
\begin{equation}\label{15}
  \|w_{\alpha}\|\leq  \frac{1}{\sqrt{\lambda_{1}}}\left(\int_{\Omega}u_{\alpha}^{2(p-1)}dx\right)^{1/2}\leq  \frac{1}{\sqrt{\lambda_{1}}}s_{\alpha}^{p-1}|\Omega|^{1/2}.
\end{equation}
  Applying \eqref{15} in \eqref{16}, we obtain
$$
  \mathcal{P}_{k}(\alpha)\leq \left(\max\limits_{s\in [0, s_{\alpha}]}\widehat{f}_{\alpha}(s)\right)
  \frac{C_{1}}{\sqrt{\lambda_{1}}}s_{\alpha}^{p-1}|\Omega|^{1/2}\quad \forall \
  \alpha\in (t_{k-1}, t_{k}).
$$
  Using condition \eqref{f_{5}} we get the conclusion of Claim 2.

\medskip

  From the continuity of $\mathcal{P}_{k}$ (see Proposition \ref{pp2}),
  Claims 1 and 2 and the intermediate value theorem for continuous real functions,
  we conclude that $\mathcal{P}_{k}$ has at least two fixed points
  $\alpha_{1,k}$ and $\alpha_{2,k}$ in the interval $(t_{k-1},t_{k})$.
  The conclusions of Claims 1 and 2 are illustrated in Figure 1.

\begin{figure}[h]
\begin{center}
\begin{picture}(300,250)(0,0)
 \put(0,30){\vector(1,0){280}} 
 \put(30,0){\vector(0,1){250}}  
  \multiput(10,10)(2,2){115}{\circle*{0.01}}

 \multiput(80,30)(0,4){55}{\line(0,2){1}}
 \multiput(200,30)(0,4){55}{\line(0,2){1}}

 \multiput(102,30)(0,2){36}{\line(0,1){1}}
    \put(102,102){\circle*{3}}
 \multiput(126,30)(0,2){10}{\line(0,1){1}}
    \put(126,49){\circle*{3}}
 \multiput(175,30)(0,2){72}{\line(0,1){1}}
    \put(175,175){\circle*{3}}

 \bezier{3000}(84,250)(110,-70)(160,120)
 \bezier{3000}(160,120)(180,200)(200,240)
   \color{white}\put(200,240){\circle*{3}}
   \color{black}\put(200,240){\circle{3}}

  \put(80,26){\makebox(0,0)[ct]{{\large $t_{k-1}$}}}
  \put(102,24){\makebox(0,0)[ct]{{\large $\alpha_{1,k}$}}}
  \put(126,24){\makebox(0,0)[ct]{{\large $\alpha_0$}}}
  \put(175,24){\makebox(0,0)[ct]{{\large $\alpha_{2,k}$}}}
  \put(202,26){\makebox(0,0)[ct]{{\large $t_k$}}}

  \put(268,34){\makebox(0,0)[lb]{{\large $\alpha$}}}

  \put(25,238){\makebox(0,0)[rc]{{\large $y$}}}

  \put(140,238){\makebox(0,0)[rc]{{\large $y=P_k(\alpha)$}}}
  \put(265,230){\makebox(0,0)[rc]{\large $y=\alpha$}}

\end{picture}

\end{center}
\caption{Properties of $\mathcal{P}_{k}$ in $(t_{k-1},t_k)$ due to Claims 1 and 2.}
\end{figure}
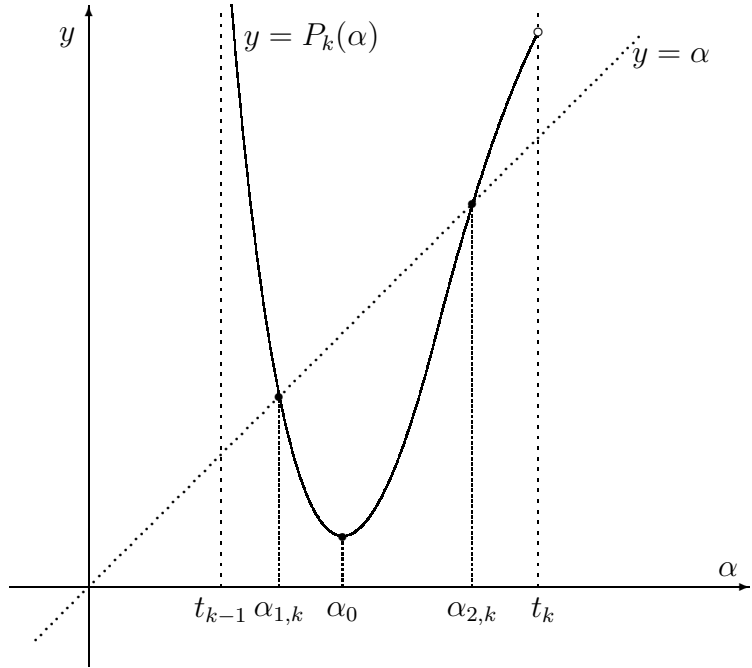

\end{proof}

\subsection{Conclusion of the proof of Theorem \ref{main1}}\label{subsec:final}

\medskip

  From Propositions \ref{pro1} and \ref{pp3} it follows
  that problem \eqref{P1} has two classical positive solutions $u_{k, 1}$
  and $u_{k, 2}$ such that
$$
  t_{k-1}<\int_{\Omega}u_{k, 1}^{p}dx<\int_{\Omega}u_{k, 2}^{p}dx<t_{k},
$$
  for any fixed $k\in \{1, \ldots, K\}$.
  This finishes the proof.\hspace*{\fill} $\square$

\section{Examples of nonlinearities $f$}\label{se:exam}\label{sec:Final}

  Let us provide some examples of function $f$ satisfying hypotheses of Theorem \ref{main1}

\medskip

\noindent
 \textbf{(a)}
  Let us take an integer $K\geq 1$,
  $U\geq (K/\int_{\Omega}e_1^pdx)^{1/p}$
  and let
  $S\colon [0,K]\to (U,+\infty)$
  be a $C^1$-function.
  We put
\[
  \ol{S}:=\max\limits_{t\in [0,K]}S(t),\quad
  \ul{S}:=\min\limits_{t\in [0,K]}S(t)>U,
\]
  and
\[
  M:=\lambda_{1}^{1/2}/(2C_{1}|\Omega|^{1/2})
 .
\]
  Let $L\colon \R\to [0,\infty)$ be a continuous function such that
  $L(w)=0$ for $w\leq 0$, $L$ is increasing and $C^1$ on $(0,\infty)$, $\lim\limits_{w\to\infty}L(w)=+\infty$,
  $L(\ul{S})>\lambda_1$, and $\max\limits_{s\in [0,\ol{S}]}s\cdot L(\ol{S}-s)<M/\ol{S}^{p-1}$.

  Now, defining
\begin{equation}\label{ex_f}
  f(s,t):=s\cdot L\bigg(\frac{S(t)-s}{|\sin\pi t|}\bigg),
\end{equation}
  we will show that $f$ satisfies hypothesis \eqref{f_{0}}-\eqref{f_{5}}, with
  $t_i=i$ for $i\in\{0,1,\ldots,K\}$ and $s_{\alpha}=S(\alpha)$.

  Indeed, $f\colon\R\times\mathcal{A}\to\R$, where
  $\mathcal{A}=(0,\infty)\setminus \{1,\ldots,K\}$,
  $f(\cdot,\alpha)\in C^1(\R)$ for all $\alpha\in \mathcal{A}$,
  $f(s_{\alpha},\alpha)=0$,
  $f(s,\alpha)>0$ for $s\in (0,s_{\alpha})$. Thus hypothesis \eqref{f_{0}} holds.
  By the assumptions on $L$, it is also obvious that \eqref{f_{1}} is valid.

  Next, we have
\[
  \frac{f(s,\alpha)}{s}=L\bigg(\frac{s_{\alpha}-s}{|\sin\pi\alpha|}\bigg)
\]
  and so we see that the map $(0, s_{\alpha})\ni s\mapsto f(s, \alpha)/s$ is decreasing (as $L$ is increasing), hence \eqref{f_{2}} holds.

  For any $k\in \{1,\ldots, K\}$ and $\alpha\in (t_{k-1},t_k)$, we have
\[
  \gamma_{\alpha}
  =\lim\limits_{s\to 0^{+}}\frac{f(s, \alpha)}{s}
  =\lim\limits_{s\to 0^{+}}L\bigg(\frac{s_{\alpha}-s}{|\sin\pi\alpha|}\bigg)
  =L\bigg(\frac{s_{\alpha}}{|\sin\pi\alpha|}\bigg)
  \geq L(s_{\alpha})
  \geq L(\ul{S})
  >\lambda_1,
\]
  thus hypothesis \eqref{f_{3}} holds.

  Hypothesis \eqref{f_{4}} 
  is satisfied, by the assumptions on the map $S$ (see the definition of $U$).

  Finally, for a fixed $k\in\{1, \ldots, K\}$, we have
\begin{eqnarray*}
  \inf\limits_{\alpha\in (t_{k-1}, t_{k})}\max\limits_{s\in[0, s_{\alpha}]} \frac{f(s,\alpha)s_{\alpha}^{p-1}}{\alpha}
  &=&
  \inf\limits_{\alpha\in (k-1, k)}\max\limits_{s\in[0, s_{\alpha}]}\frac{ss_{\alpha}^{p-1}}{\alpha}L\bigg(\frac{S(\alpha)-s}{|\sin\pi\alpha|}\bigg)\\
  & \leq &
    \max\limits_{s\in[0, \ol{S}]}\frac{s\ol{S}^{p-1}}{k-\frac{1}{2}}L\bigg(\frac{S(k-\frac{1}{2})-s}{|\sin\pi/2|}\bigg) \\
  &\leq&
    2\ol{S}^{p-1}\max\limits_{s\in[0, \ol{S}]}s L(\ol{S}-s)\\
  & < &
    2\ol{S}^{p-1}M/\ol{S}^{p-1}
  =
   2M
  =
   \frac{\lambda_{1}^{1/2}}{C_{1}|\Omega|^{1/2}}
\end{eqnarray*}
  thus hypothesis \eqref{f_{5}}
  holds.

\medskip

\noindent
 \textbf{(b)}
  Let us now provide some particular examples of functions $S$ and $L$ satisfying the above assumptions.
  Take $K$, $U$, and $M$ as defined in \textbf{(a)} and put
 \[
  A:=U+1,\quad
  n>\frac{(\lambda_1+1)(A+1)^p}{M}-1,\quad\textrm{and}\quad
  \beta:=\frac{\lambda_1+1}{A^n}.
 \]
   Next choose $B>A$ such that
 \[
   B<\min\left\{A\sqrt[n]{\frac{(n+1)M}{(\lambda_1+1)(A+1)^p}},\ A+1\right\}.
 \]
  Now we define function $S\colon [0,K]\to (0,+\infty)$, by
\[
  S(t):=\frac{B-A}{K}t+A\quad\forall t\in [0,K],
\]
  and function $L\colon \R\to [0,\infty)$, by
\[
  L(w):=\left\{
  \begin{array}{lll}
    \beta w^n & \textrm{for} & w\geq 0,\\
    0         & \textrm{for} & w<0.
  \end{array}
  \right.
\]
  It is easy to check, that $\ul{S}=A>U$ and $\ol{S}=B$.
  Moreover  the function $L$ is continuous,
  increasing and $C^1$ on $(0,\infty)$, $\lim\limits_{w\to\infty}L(w)=+\infty$,
\[
  L(\ul{S})
  =
  L(A)
  =
  \beta A^n
  =
  \frac{\lambda_1+1}{A^n}A^n
  =
  \lambda_1+1
  >
  \lambda_1,
\]
  and finally the maximum of the function $[0,\ol{S}]\ni s\longmapsto s\cdot L(\ol{S}-s)=\beta s(B-s)^n$
  is attained at $s_{max}=\frac{B}{1+n}$ and so
\begin{eqnarray*}
  \max\limits_{s\in [0,\ol{S}]}s\cdot L(\ol{S}-s)
 & =&
  \beta\frac{B}{1+n}\bigg(\frac{nB}{1+n}\bigg)^n \\
&  =&\frac{\lambda_1+1}{1+n}\bigg(\frac{B}{A}\bigg)^n B \bigg(\frac{n}{1+n}\bigg)^n\\
  & < &
    \frac{\lambda_1+1}{1+n}\frac{(n+1)M}{(\lambda_1+1)(A+1)^p}(A+1)\\
&  =& M/(A+1)^{p-1}
  <
  M/B^{p-1}=M/\ol{S}^{p-1},
\end{eqnarray*}
  by the definition of $B$. Thus both functions $S$ and $L$ fit into the framework of example \textbf{(a)},
  and so the function $f$ defined by \eqref{ex_f} satisfies hypotheses \eqref{f_{0}}-\eqref{f_{5}}.

\medskip

\noindent
 \textbf{(c)}
  For another choice of functions $S$ and $L$ satisfying the assumptions of \textbf{(a)}, take again $K$, $U$, and $M$ as in \textbf{(a)} and put
  $A:=U+1$. Next fix two numbers $B$ and $D$ such that
\[
   A<B<D<\min\left\{A+\frac{MA}{(\lambda_1+1)(A+1)^p},\ A+1\right\}
\]
  and define
\[
  C:=\frac{(\lambda_1+1)(D-A)}{A}.
\]
  Let $S\colon [0,K]\to [A,B]$ be any $C^1$-function.
  Then $\ul{S}\geq A$ and $\ol{S}\leq B$.

  On the interval $[0,B]$ we defined an increasing $C^1$-function
\[
  L(w)=-C\frac{w}{w-D}
\]
  and extend it outside $[0,B]$ in such a way that it is still $C^1$ and increasing.
  First, we have
\[
  L(\ul{S})\ge L(A)=-\frac{(\lambda_1+1)(D-A)}{A}\cdot\frac{A}{A-D}=\lambda_1+1>\lambda_1.
\]
  Next
\[
  \max_{s\in [0,\ol{S}]}s\cdot L(\ol{S}-s)
  \leq
    \max_{s\in [0,B]}s\cdot L(B-s)
  =
    \max_{s\in [0,B]} -Cs\frac{B-s}{B-s-D}
  =
    \max_{s\in [0,B]} \frac{Cs(B-s)}{s+(D-B)}.
\]
  The last maximum is attained at
  $s_{max}=\sqrt{D(D-B)}-(D-B)\in (0,B)$,  with the value
\begin{eqnarray*}
  \max_{s\in [0,B]} \frac{Cs(B-s)}{s+(D-B)}
    &=&\frac{C(\sqrt{D(D-B)}-(D-B))(D-\sqrt{D(D-B)})}{\sqrt{D(D-B)}}\\
  & = &
    \frac{(\lambda_1+1)(D-A)}{A}\bigg(1-\sqrt{\frac{D-B}{D}}\bigg)(D-\sqrt{D(D-B)})\\
  &\le&
    \frac{(\lambda_1+1)}{A}(D-A)D\\
  & < &
    \frac{(\lambda_1+1)}{A}\frac{MA}{(\lambda_1+1)(A+1)^p}(A+1)\\
  &=&\frac{M}{(A+1)^{p-1}}
  \leq\frac{M}{B^{p-1}}
  \leq\frac{M}{\ol{S}^{p-1}}.
\end{eqnarray*}
  Thus both functions $S$ and $L$ fit into the framework of example \textbf{(a)}.


\begin{thebibliography}{99}


\bibitem{Ag}
  S. Agmon,
  The $L_{p}$ approach to the Dirichlet problem,
  {\it Ann. Scuola Norm. Sup. Pisa},
  {\bf13}, (1959), 405--448.

\bibitem{AA1}
  A. Ambrosetti and D. Arcoya,
  Positive solutions of elliptic Kirchhoff equations,
  {\it Advanced Nonlinear Studies},
  {\bf 17}, n. 1, (2017) 3--16.

\bibitem{AA2}
  A. Ambrosetti and D. Arcoya,
  Remarks on non homogeneous elliptic Kirchhoff equations,
  {\it Nonlinear Differ. Equ. Appl.},
  {\bf 23}, (2016) Art. 57. https://doi.org/10.1007/s00030-016-0410-1.

\bibitem{ABC}
  A. Ambrosetti, H. Brezis, and G. Cerami,
  Combined effects of concave and convex nonlinearities in some elliptic problems.
  {\it J. Funct. Anal.},
  {\bf122}, (1994), 519--543.

\bibitem{AmPr}
  A. Ambrosetti and G. Prodi,
  A Primer of Nonlinear Analysis,
  Cambridge University Press, 1993.

\bibitem{BO}
  H. Brezis and L. Oswald,
  Remarks on sublinear elliptic equations,
  {\it Nonlinear Anal.},
  {\bf 10}, n. 1, (1986), 55--64.

\bibitem{Car}
  G.F. Carrier,
  On the non-linear vibration problem of the elastic string,
  {\it Q. J. Appl. Math.},
  {\bf 3}, (1945), 151--165.

\bibitem{CR}
  M. Chipot and J.F. Rodrigues,
  On a class of nonlocal nonlinear elliptic problems,
  {\it RAIRO - Mod\'elisation math\'ematique et analyse num\'erique}
  {\bf 26}, n. 3, (1992) 447--467.

\bibitem{DMSS}
  M. Delgado, C. Morales-Rodrigo, J.R. Santos J\'unior, and A. Su\'arez,
  Non-local degenerate diffusion coefficients break down the components of positive solutions,
  {\it Adv. Nonlinear Stud.},
  {\bf 20}, n. 1, (2019), 19--30.

\bibitem{FG}
  J. Furter and M. Grinfeld,
  Local vs. nonlocal interactions in population dynamics,
  {\it J. Math. Biol.},
  {\bf 27}, (1989), 65--80.

\bibitem{FMS}
  T. Figueiredo-Sousa, C. Morales-Rodrigo, and A. Su\'arez,
  A non-local non-autonomous diffusion problem: linear and sublinear cases,
  {\it Z. Angew. Math. Phys.},
  {\bf 68}, (2017), no. 5, Art. 108, 20 pp.

\bibitem{GT}
  D. Gilbarg and N.S. Trudinger,
  Elliptic Partial Differential Equations of Second Order,
  {\it Springer-Verlag}, Berlin, 1983.

\bibitem{GS1}
  L. Gasi\'nski and J. R. Santos J\'unior,
  Multiplicity of positive solutions for an equations with degenerate nonlocal diffusion,
  {\it Comput. Math. Appl.},
  {\bf 78}, (2019), 136--143.

\bibitem{GS2}
  L. Gasi\'nski and J. R. Santos J\'unior,
  Nonexistence and multiplicity of positive solutions for an equation with degenerate nonlocal diffusion,
  to appear in {\it Bull. London Math. Soc.}

\bibitem{Kir}
  G. Kirchhoff,
  Mechanik,
  {\it Teubner}, Leipzig, 1883.

\bibitem{SS}
  J.R. Santos J\'unior and G. Siciliano,
  Positive solutions for a Kirchhoff problem with vanishing nonlocal term,
  {\it J. Differential Equations},
  {\bf 265}, (2018), 2034--2043.


\end{thebibliography}
\end{document}